\documentclass[11pt,letterpaper]{amsart} 
\usepackage[rgb]{xcolor}
\usepackage{tikz}
\usepackage{verbatim}
\usepackage{amsmath}
\usepackage{amsxtra}
\usepackage{amscd}
\usepackage{amsthm}
\usepackage{amsfonts}
\usepackage{amssymb}
\usepackage{eucal}
\usepackage{float}
\usepackage{mathrsfs}
\usepackage{graphicx}
\usepackage{stmaryrd}
\usepackage{tikz-cd}
\usetikzlibrary{arrows}
\usetikzlibrary{cd}
\usepackage{ytableau}
\usepackage{tikz-cd}
\usetikzlibrary{arrows}
\usetikzlibrary{cd}
\tikzcdset{every label/.append style = {font = \small}}
\allowdisplaybreaks

\usepackage{latexsym,todonotes}

\PassOptionsToPackage{hyphens}{url}\usepackage[linktocpage=true]{hyperref}
\hypersetup{colorlinks,linkcolor=blue,urlcolor=orange,citecolor=blue}
\usepackage[all, cmtip]{xy}

\usepackage{stmaryrd}
\usepackage{color} 

\usepackage{amsrefs}

\newtheorem{thm}[equation]{Theorem} 
\newtheorem{thm2}[equation]{Theorem$^\prime$}

\newtheorem{lem}[equation]{Lemma}
\newtheorem{prop}[equation]{Proposition}

\theoremstyle{definition}

\theoremstyle{remark}
\newtheorem{rem}[equation]{Remark}

\numberwithin{equation}{section}

\begin{document}

\newcommand{\thmref}[1]{Theorem~\ref{#1}}
\newcommand{\secref}[1]{Section~\ref{#1}}
\newcommand{\lemref}[1]{Lemma~\ref{#1}}
\newcommand{\propref}[1]{Proposition~\ref{#1}}
\newcommand{\corref}[1]{Corollary~\ref{#1}}
\newcommand{\remref}[1]{Remark~\ref{#1}}
\newcommand{\eqnref}[1]{(\ref{#1})}

\newcommand{\exref}[1]{Example~\ref{#1}}

 \newcommand{\GSp}{\mathrm{GSp}}
 \newcommand{\PGSp}{\mathrm{PGSp}}
\newcommand{\PGSO}{\mathrm{PGSO}}
\newcommand{\PGO}{\mathrm{PGO}}
\newcommand{\SO}{\mathrm{SO}}
\newcommand{\GO}{\mathrm{GO}}
\newcommand{\GSO}{\mathrm{GSO}}
\newcommand{\Spin}{\mathrm{Spin}}
\newcommand{\Sp}{\mathrm{Sp}}
\newcommand{\PGL}{\mathrm{PGL}}
\newcommand{\GL}{\mathrm{GL}}
\newcommand{\SL}{\mathrm{SL}}
\newcommand{\U}{\mathrm{U}}
\newcommand{\ind}{\mathrm{ind}}
\newcommand{\Ind}{\mathrm{Ind}}
\newcommand{\im}{\mathrm{im}}
\renewcommand{\ker}{\mathrm{ker}}
 \newcommand{\triv}{\mathrm{triv}}
  \newcommand{\std}{\mathrm{std}}
 \newcommand{\Ad}{\mathrm{Ad}}
  \newcommand{\ad}{\mathrm{ad}}
  \newcommand{\Tr}{\mathrm{Tr}}
  \renewcommand{\S}{\mathscr{S}}
  \newcommand{\Y}{\mathbb{Y}}
\newcommand{\End}{\mathrm{End}}
\newcommand{\vol}{\mathrm{vol}}
\newcommand{\bigzero}{\mbox{\normalfont\Large\bfseries 0}}

\newtheorem{innercustomthm}{{\bf Theorem}}
\newenvironment{customthm}[1]
  {\renewcommand\theinnercustomthm{#1}\innercustomthm}
  {\endinnercustomthm}
  
  \newtheorem{innercustomcor}{{\bf Corollary}}
\newenvironment{customcor}[1]
  {\renewcommand\theinnercustomcor{#1}\innercustomcor}
  {\endinnercustomthm}
  
  \newtheorem{innercustomprop}{{\bf Proposition}}
\newenvironment{customprop}[1]
  {\renewcommand\theinnercustomprop{#1}\innercustomprop}
  {\endinnercustomthm}

\newcommand{\bbinom}[2]{\begin{bmatrix}#1 \\ #2\end{bmatrix}}
\newcommand{\cbinom}[2]{\set{\^!\^!\^!\begin{array}{c} #1 \\ #2\end{array}\^!\^!\^!}}
\newcommand{\abinom}[2]{\ang{\^!\^!\^!\begin{array}{c} #1 \\ #2\end{array}\^!\^!\^!}}
\newcommand{\qfact}[1]{[#1]^^!}

\newcommand{\nc}{\newcommand}

\nc{\Ord}{\text{Ord}_v}

 \nc{\A}{\mathbb A} 
  \nc{\G}{\mathbb G} 
\nc{\Ainv}{\A^{\rm inv}}
\nc{\aA}{{}_\A}
\nc{\aAp}{{}_\A'}
\nc{\aff}{{}_\A\f}
\nc{\aL}{{}_\A L}
\nc{\aM}{{}_\A M}
\nc{\Bin}{B_i^{(n)}}
\nc{\dL}{{}^\omega L}
\nc{\Z}{{\mathbb Z}}
 \nc{\C}{{\mathbb C}}
 \nc{\N}{{\mathbb N}}
 \nc{\R}{{\mathbb R}}
 \nc{\fZ}{{\mf Z}}
 \nc{\F}{{\mf F}}
 \nc{\Q}{\mathbb{Q}}
 \nc{\la}{\lambda}
 \nc{\ep}{\epsilon}
 \nc{\h}{\mathfrak h}
 \nc{\He}{\bold{H}}
 \nc{\htt}{\text{tr }}
 \nc{\n}{\mf n}
 \nc{\g}{{\mathfrak g}}
  \renewcommand{\o}{{\mathfrak o}}
 \nc{\DG}{\widetilde{\mathfrak g}}
 \nc{\SG}{\breve{\mathfrak g}}
 \nc{\is}{{\mathbf i}}
 \nc{\V}{\mf V}
 \nc{\bi}{\bibitem}
 \nc{\E}{\mc E}
 \nc{\ba}{\tilde{\pa}}
 \nc{\half}{\frac{1}{2}}
 \nc{\hgt}{\text{ht}}
 \nc{\ka}{\kappa}
 \nc{\mc}{\mathcal}
 \nc{\mf}{\mathfrak} 
 \nc{\hf}{\frac{1}{2}}
\nc{\ov}{\overline}
\nc{\ul}{\underline}
\nc{\I}{\mathbb{I}}
\nc{\xx}{{\mf x}}
\nc{\id}{\text{id}}
\nc{\one}{\bold{1}}
\nc{\mfsl}{\mf{sl}}
\nc{\mfgl}{\mf{gl}}
\nc{\ti}[1]{\textit{#1}}
\nc{\Hom}{\mathrm{Hom}}
\nc{\Irr}{\mathrm{Irr}}
\nc{\Cat}{\mathscr{C}}
\nc{\CatO}{\mathscr{O}}
\renewcommand{\O}{\mathrm{O}}
\nc{\Tan}{\mathscr{T}}
\nc{\Umod}{\mathscr{U}}
\nc{\Func}{\mathscr{F}}
\nc{\Kh}{\text{Kh}}
\nc{\Khb}[1]{\llbracket #1 \rrbracket}

\nc{\ua}{\mf{u}}
\nc{\nb}{u}
\nc{\inv}{\theta}
\nc{\mA}{\mathcal{A}}
\newcommand{\TT}{\mathbf T}
\newcommand{\TA}{{}_\A{\TT}}
\newcommand{\tK}{\widetilde{K}}
\newcommand{\al}{\alpha}
\newcommand{\Fr}{\bold{Fr}}

\nc{\Qq}{\Q(v)}
\nc{\uu}{\mathfrak{u}}
\nc{\Udot}{\dot{\U}}

\nc{\f}{\bold{f}}
\nc{\fprime}{\bold{'f}}
\nc{\B}{\bold{B}}
\nc{\Bdot}{\dot{\B}}
\nc{\Dupsilon}{\Upsilon^{\vartriangle}}
\newcommand{\T}{\texttt T}
\newcommand{\vs}{\varsigma}
\newcommand{\Pa}{{\bf{P}}}
\newcommand{\Padot}{\dot{\bf{P}}}

\nc{\ipsi}{\psi_{\imath}}
\nc{\Ui}{{\bold{U}^{\imath}}}
\nc{\uidot}{\dot{\mathfrak{u}}^{\imath}}
\nc{\Uidot}{\dot{\bold{U}}^{\imath}}
 \nc{\be}{e}
 \nc{\bff}{f}
 \nc{\bk}{k}
 \nc{\bt}{t}
 \nc{\bs}{\backslash}
 \nc{\BLambda}{{\Lambda_{\inv}}}
\nc{\Ktilde}{\widetilde{K}}
\nc{\bktilde}{\widetilde{k}}
\nc{\Yi}{Y^{w_0}}
\nc{\bunlambda}{\Lambda^\imath}
\newcommand{\Iwhite}{\I_{\circ}}
\nc{\ile}{\le_\imath}
\nc{\il}{<_{\imath}}

\newcommand{\ff}{B}


\nc{\etab}{\eta^{\bullet}}
\newcommand{\Iblack}{\I_{\bullet}}
\newcommand{\wb}{w_\bullet}
\newcommand{\UIblack}{\U_{\Iblack}}

\newcommand{\blue}[1]{{\color{blue}#1}}
\newcommand{\red}[1]{{\color{red}#1}}
\newcommand{\green}[1]{{\color{green}#1}}
\newcommand{\white}[1]{{\color{white}#1}}

\newcommand{\dvd}[1]{t_{\odd}^{{(#1)}}}
\newcommand{\dvp}[1]{t_{\ev}^{{(#1)}}}
\newcommand{\ev}{\mathrm{ev}}
\newcommand{\odd}{\mathrm{odd}}

\newcommand\TikCircle[1][2.5]{{\mathop{\tikz[baseline=-#1]{\draw[thick](0,0)circle[radius=#1mm];}}}}

\newcommand{\commentcustom}[1]{}

\raggedbottom

\title[Whittaker models and relative Langlands duality II]
{Generalised Whittaker models as instances of relative Langlands duality II: \\
Plancherel density and global periods}

\author{Wee Teck Gan}
\author{Bryan Wang Peng Jun}
 \address{Department of Mathematics, National University of Singapore, 10 Lower Kent Ridge Road, Singapore 119076}
\email{matgwt@nus.edu.sg}
\email{bwang@nus.edu.sg}

\begin{abstract}
In an earlier paper of the authors, a general family of instances of the relative Langlands duality of Ben-Zvi$-$Sakellaridis$-$Venkatesh \cite{BZSV} were proposed and studied in the setting of branching problems for smooth representations. In this paper, we show the numerical conjectures of \cite{BZSV} for the local Plancherel density, as well as an application to their conjectures on global periods, for this general family of instances. 

\end{abstract}

\maketitle

\setcounter{tocdepth}{1}
\tableofcontents

\section{Introduction}\label{sec:Introduction}

This is a sequel to an earlier paper \cite{GW} of the authors. There, the main aim was to explicate the recently-proposed relative Langlands duality of \cite{BZSV} in the (local) setting of branching problems for smooth representations, and to propose many (new) cases in which such duality could be verified. 

However, \cite{GW} only treated the conjectures of \cite{BZSV} in the smooth setting, and did not treat some other conjectures there, such as  the conjectural precise formula for the local Plancherel density \cite[Section 9]{BZSV} as well as the conjecture about global period \cite[Section 14]{BZSV}.
The aim of the present paper is to treat these remaining issues in one representative new case introduced in \cite[Section 7]{GW}. The method employed can be similarly applied without much difficulty to the other cases (but we limit to this case for ease of exposition), using the theta correspondence between the odd orthogonal group and metaplectic groups, and the Lapid-Mao conjecture
for metaplectic groups (in the global case).

\vskip 5pt

\subsection{Background} Let us now recall the relevant context and background as laid out in \cite{BZSV} and \cite{GW}. One of the central themes of the relative Langlands program is to characterize the non-vanishing of certain periods of automorphic forms on a reductive group $G$, and when the period is non-zero, to relate it  to certain special values of automorphic $L$-functions. For a long time, however, this relation between periods and $L$-functions has been relatively poorly understood, with a disparate collection of instances known, but no systematic framework with which to understand this relationship. 

A systematic (i.e. general) study of the relations between periods and $L$-functions was first laid out in the book \cite{SV}, using the framework of spherical varieties. There, the important observation was further made, that the factorisation of the global period functionals (associated to a spherical variety $X$), into corresponding local functionals, should be facilitated by local Plancherel functionals obtained from the spectral decomposition of $L^2(X)$. In particular, at unramified places, the evaluation of these Plancherel functionals on the (unit) spherical vector should give rise to (special values of) corresponding local $L$-factors. This forms the basis for relating the global period functional to a corresponding $L$-value.

However, the framework of \cite{SV} was still not entirely satisfactory in some respects. On one hand, as also explained in \cite{GW}, there are several important examples exhibiting analogous phenomena, such as models with a Whittaker-twisted component (arising from a nilpotent orbit), or Howe duality/theta correspondence, which do not fall (strictly) into the framework of \cite{SV}. On the other hand, the framework of \cite{SV} also does not explain which (and why) a particular $L$-function should be related to a given period, so that the period-$L$-function relationship still remained relatively mysterious. 

The framework recently proposed in \cite{BZSV} addresses these issues. On one hand, the periods should be indexed by certain Hamiltonian $G$-varieties, known as hyperspherical varieties (generalising spherical varieties), which encompass the aforementioned examples of Whittaker-twisted models and theta correspondence. On the other hand, the corresponding $L$-functions should also be indexed by analogous hyperspherical $G^{\vee}$-varieties (where $G^{\vee}$ is the Langlands dual group of $G$). Key to this is the insight that there should exist an involutive \textit{duality} of hyperspherical varieties (in the spirit of Langlands duality), so that the period associated to a hyperspherical $G$-variety $M_1$ is related to the $L$-function associated to its dual hyperspherical $G^{\vee}$-variety $M_2$, and \textit{vice versa}. Furthermore, a similar structure should exist at the local (unramified) level: the Plancherel density associated to (the quantization of) $M_1$ should be related to the local $L$-factor associated to $M_2$, and vice versa. 
\vskip 5pt

\subsection{Results of \cite{GW}}\label{results_gw}
Let us recall briefly the results of the prequel \cite{GW}.
For this, we first recall  that a hyperspherical $G$-variety $M$ can  be built out of the following two pieces of data:
\vskip 5pt

\begin{itemize}
\item a morphism $\iota: H \times \SL_2 \longrightarrow G$, such that $H \hookrightarrow Z_G(\iota(\SL_2))$ is a spherical subgroup;
\item a symplectic representation $S$ of $H$.
\end{itemize}
In \cite{GW}, we consider special cases of hyperspherical varieties where:
\begin{itemize}
\item $G$ is a symplectic or orthogonal group;
\item $H = Z_G(\iota(\SL_2))$;
\item $S=0$.
\end{itemize}
Under these hypotheses, the associated hyperspherical $G$-variety $M$ just depends on a unipotent conjugacy class $e$ in $G$.
In this setting, we classified in \cite[Section 5]{GW} those $e$ for which $M = M_e$ hyperspherical.  In particular, if $e$ corresponds to partitions of hook type, then $M_e$ is hyperspherical. For this general family of examples, we determined in \cite{GW} their hyperspherical dual $M_e^{\vee}$, by  describing the smooth local period problem associated to $M_e$ in terms of the geometry of a dual hyperspherical $G^{\vee}$-variety $M_e^{\vee}$. Indeed,  $M_e^{\vee}$ is of the same type as $M_e$, namely is associated to a nilpotent orbit in $G^{\vee}$ of hook type, so that
\[  M_e^{\vee} = M_{e^{\vee}} \quad \text{for some $e^{\vee} \in G^{\vee}$.} \]
Moreover, we determine $e^{\vee}$ in terms of $e$. 
 \vskip 10pt
 
With these notions introduced,  we now specify  the representative family of examples from \cite[Section 7.1]{GW} that we will work with in this paper:
\vskip 5pt
\begin{itemize}
\item $G = \O_{2k}$ is an even orthogonal group;
\item the nilpotent conjugacy class $e^\vee \in G^\vee$ corresponds to the partition $[2k-2a-1,1^{2a+1}]$, so that the data defining the hyperspherical $G^\vee$-variety $M_1= M_{e^\vee}$ consists of the datum 
\[ \iota_1: \O_{2a+1}\times \SL_2\rightarrow \O_{2k}, \quad \text{with $S = 0$;} \]
\item the nilpotent conjugacy class $e \in G = \O_{2k}$ corresponds to the partition $[2a-1, 1^{2k-2a+1}]$, so that the data defining the hyperspherical $G$-variety $M_2 = M_{e}$ consists of the datum
\[ \iota_2: \O_{2k-2a+1}\times \SL_2\rightarrow \O_{2k}, \quad \text{  with $S = 0$.} \]
\end{itemize}
Then, for this particular representative family,  the main result in \cite{GW} can be loosely formulated as:

\begin{thm2}\label{thm2:EvenOrthogonal} \cite[Section 7.1]{GW}
The hyperspherical varieties $M_1$ and $M_2$ defined above are dual under the expected \cite{BZSV}-duality of hyperspherical varieties. 
 More precisely:
 \vskip 5pt
 
 \begin{itemize}
 \item[(i)] the irreducible smooth representations (of Arthur type) of $G^\vee = \O_{2k}$ which occur as submodules in the (smooth) quantization of $M_1$ have  A-parameters  factoring through $\iota_2$ and can be constructed as local theta lifts from generic representations of $\Sp_{2k-2a}$;
 \item[(ii)] the irreducible smooth representations (of Arthur type) of $G = \O_{2k}$ which occur as submodules in the (smooth) quantization of $M_2$ have  A-parameters  factoring through $\iota_1$ and can be constructed as local theta lifts from generic representations of $\Sp_{2a}$. 
 \end{itemize}
    \end{thm2}

 However, as mentioned earlier, the corresponding conjectures  in \cite{BZSV} concerning the Plancherel decomposition and global periods were not studied in \cite{GW}, and it is the aim of this paper to address these problems.

 \subsection{Main result}
The main result of the paper is (Theorem \ref{thm:MainThm}):

\begin{thm}
    The numerical conjecture for the local Plancherel density \cite[Proposition 9.2.1]{BZSV} is true for the hyperspherical dual pair $(M_1,M_2)$. 
    
    In other words, the local Plancherel density associated to (the quantization of) $M_2$ corresponds to a local $L$-value attached to $M_1$, and vice versa. 
    \end{thm}

We now explain the precise meaning of this last statement:
\[  \text{`the Plancherel density associated to $M_2$ is related to a local $L$-value attached to $M_1$'.} \]
 Denote $G=\O_{2k}$, so that $M_2$ is a $G$-variety, and $M_1$ is a $G^\vee$-variety. Note first that $M_2$ is a twisted cotangent bundle of a certain $G$-variety \[ X = \O_{2k-2a+1}U,\overline{\chi_\gamma}\bs G, \] for $U$ a unipotent subgroup of $G$ (corresponding to the nilpotent orbit defining $M_2$) and $\chi_\gamma$ a unitary character of $U$ (we will define $X$ more precisely in Section \ref{sec:Notation}).

Our first main result (Theorem \ref{thm:MainSpectral2}) concerns the spectral decomposition of $L^2(X)$, which is the quantization of $M_2$, and shows that the spectral measure on $L^2(X)$ is given by the pushforward, via theta lifting, of the (Whittaker-)Plancherel measure on $G_X := \Sp_{2a}$, or in other words, via (roughly speaking) functorial lifting through the map \[ \O_{2a+1}\times \SL_2\rightarrow \O_{2k} \] defining $M_1$. At least on the level of unramified representations, the theta lifting is certainly known to facilitate such a functorial lifting, as we already saw in Theorem' \ref{thm2:EvenOrthogonal}. In other words, Theorem \ref{thm:MainSpectral2} is the $L^2$-version of Theorem' \ref{thm2:EvenOrthogonal}.

Our second main result (Theorem \ref{thm:MainThm2}) concerns the Plancherel density of the basic function $\phi_0$ on $X$ (coming from the spectral measure of $L^2(X)$). This is a measure on $\widehat{G}_{\text{unramified}}$ (the unramified unitary dual of $G$), and hence (via 
 the Satake isomorphism) identified with a measure on $A^\vee/W$, where $A^\vee$ is the maximal torus in $\O_{2k}$ and $W$ its Weyl group. 
We would  like to describe this measure in terms of the data of $M_1$.

Denote $G_X^\vee = \O_{2a+1}$, and denote by  $A_X^{\vee (1)}$  the maximal compact subgroup of the torus $A_X^\vee$ in $G_X^\vee$. 
In view of the first main result, this measure should be given as a pushforward of a certain measure on $A_X^{\vee (1)}$, via the map \begin{align*}
        A_X^{\vee (1)} &\rightarrow A^\vee \\
        (-) &\mapsto (-)\cdot q^{-\rho_{L(X)}},
    \end{align*}
where here $\rho_{L(X)}$ is $\frac{1}{2}$ of the (co)weight for $\O_{2k}$ corresponding to the map $\mfsl_2\rightarrow \mf{so}_{2k}$ (coming from the map $ \O_{2a+1}\times \SL_2\rightarrow \O_{2k} $ defining $M_1$) and $q$ is the cardinality of the residue field. 

To describe the measure on $A_X^{\vee (1)}$, we recall that we have an isomorphism \[ M_1 \cong V_X \times^{G_X^\vee} G^\vee, \] for $V_X$ a finite-dimensional $\Z$-graded representation of $G_X^\vee$ \cite[Section 4.5]{BZSV}. We view $V_X$ as a $G_X^\vee \times \G_m$-representation, with $\G_m$ acting via the $\Z$-grading. The symplectic representation $S$ in the definition of $M_1$ is a direct summand of $V_X$ (in degree 1) in general. 
    
    The second main result is then that this measure on $A_X^{\vee (1)}$ is given by \[ \frac{1}{|W_X|} \frac{ \det (I - t|_{\g_X^\vee/\mf{a}_X^\vee}) }{\det(I-(t,q^{-1/2})|_{V_X}) } dt, \] for $dt$ the usual Haar measure on $A_X^{\vee (1)}$.

In more informal terms, the local unramified Plancherel density associated to $M_2$ is therefore related to a local $L$-value \[ \prod_{i\in\Z} L(i/2, V_X^{(i)})\] associated to $M_1$.

\vskip 5pt

\subsection{Method of proof} We now say a few words on the method of proof. The main tool employed, as in \cite{GW}, is that of the theta correspondence, in particular the $L^2$-theta correspondence \cite{Sa}. A rather general framework for addressing such problems concerning Plancherel decomposition (and global periods), using theta correspondence, was first laid out in a paper of the first author and Wan \cite{GW2}, and indeed we largely follow the general framework there. 

However, in \cite{GW2}, the relevant spectral decomposition is shown using a different method (first done in \cite{GG}), which, while the natural argument in that context, does not generalise to the more general cases we consider. Our approach in this paper differs in this key regard; we show the spectral decomposition in a way which is not only more general, but which is also parallel to the treatment of the smooth local and global aspects of the problem. This parallel in the treatment of the smooth local, $L^2$-local and global aspects via theta correspondence was not previously noticed in \cite{GW2}, and is in a sense the main new contribution of this paper. Furthermore, the generality of the method means that it can be similarly applied without much difficulty to the other cases considered in \cite[Section 7]{GW}, and may potentially find application to other new cases beyond those considered in \cite{GW} (since, for instance, it may be easily adapted to the case of unitary dual pairs). 

\vskip 5pt
    
\subsection{Global periods} Finally, as an application of the local (unramified) results shown in this paper, and continuing with the general framework of \cite{GW2} in the global setting, one is also able to show the numerical conjecture of \cite{BZSV} concerning global periods, as stated in the simplest case (i.e. that considered in \cite{SV}) where the $\SL_2$-type of the automorphic form and the hyperspherical (dual) space coincide, for automorphic forms on $G$ which are cuspidal and in the image of global theta lifting from (cuspidal) automorphic representations of the relevant dual group $G_X$. We refer to Theorem \ref{thm:GlobalPeriod} for the precise statement. 

\subsubsection{Degenerate Whittaker periods} In \cite{MWZ}, a relative trace formula approach was proposed to approach the aforementioned numerical conjecture of \cite{BZSV} concerning global periods (stated in \cite{MWZ} as Conjecture 1.1). In doing so, they proposed several key conjectures \cite[Conjectures 1.4, 1.8, 1.11]{MWZ} concerning the so-called ``degenerate Whittaker periods", which in the language of this paper (Section \ref{results_gw}) corresponds to the period obtained when $H$ is taken to be the trivial group. 

One of the aims of this paper is to also show how these conjectures can similarly be addressed effectively via the theta correspondence, at least in the `hook type' case considered here (but again, we expect the methods to carry over to other similar cases which can be addressed via theta correspondence). This includes the rank 1 spherical varieties $\O_{2k-1} \backslash \O_{2k}$ as special cases, which hence also addresses \cite[Conjecture 1.11]{MWZ}. This will be done in Section \ref{degenerate_whittaker} and we refer the reader there for the precise statements. 

A key obstacle in addressing these conjectures is that of defining the local relative characters for the degenerate Whittaker periods \cite[Remark 1.5]{MWZ}. We will see that the `local relative characters' are in fact furnished very naturally in the theta correspondence approach.

\vskip 5pt

In conclusion, this paper and  its predecessor \cite{GW} have achieved the full analog of the results of \cite{GW2} (shown for the rank 1 spherical variety $\O_k \backslash \O_{k+1}$) in the context of the hyperspherical varieties of $\O_{2k}$ associated to hook type partitions. 

\vskip 5pt

\subsection{Acknowledgements} We would like to thank Yiannis Sakellaridis and Akshay Venkatesh for illuminating discussions about \cite{BZSV} during the course of this work. The first author thanks Chen Wan for bringing the conjectures of \cite{MWZ} to his attention at the Tianyuan Conference 2024. The second author would also like to thank Nhat Hoang Le for some helpful conversations while the work on this paper was ongoing.  We also thank the referee for his/her careful reading of our article and for pointing out a couple of inaccuracies. 
W.T. Gan is partially supported by a Singapore government MOE Tier 1 grant R-146-000-320-114 and a Tan Chin Tuan Centennial Professorship.

\section{Preliminaries}\label{sec:Notation}

Throughout, let $F$ be a non-Archimedean local field (of characteristic 0) with residual characteristic not 2, with ring of integers $\o_F$ and uniformizer $\varpi_F$. Denote by $\kappa_F$ the residue field of $F$ and $q$ its cardinality. Fix a non-trivial unitary character $\psi : F\rightarrow \C^\times$ with conductor $\o_F$. Then $\psi$ determines an additive self-dual Haar measure on $F$, under which $\o_F$ has volume 1, which we fix throughout. 

In this paper, all classical groups are split, unless otherwise stated. 

\subsection{The case under consideration} Recall that the main aim of this paper is to show the following duality, conjectured in \cite[Section 7]{GW}, in the setting of the local (unramified) Plancherel measure, as explicated in \cite[Proposition 9.2.1]{BZSV}.

\begin{thm2}\label{thm2:EvenOrthogonal2} \cite[Section 7.1]{GW}
The hyperspherical varieties $M_1,M_2$ defined respectively by 
\begin{itemize}
    \item the datum $\O_{2a+1}\times \SL_2\rightarrow \O_{2k}$, corresponding to the nilpotent orbit with partition $[2k-2a-1,1^{2a+1}]$, and trivial $S$;\\
    
    \item and the datum $\O_{2k-2a+1}\times \SL_2\rightarrow \O_{2k}$, corresponding to the nilpotent orbit with partition $[2a-1, 1^{2k-2a+1}]$, and trivial $S$.
\end{itemize}
are dual under the expected \cite{BZSV}-duality of hyperspherical varieties. 
    
\end{thm2}

In this paper, we will focus on the Plancherel measure associated to $M_2$ (and describe it in terms of $M_1$); the dual problem involving the Plancherel measure associated to $M_1$ is completely similar. 

Similarly, the other cases considered in \cite[Section 7]{GW} can be handled in a similar manner.

\subsection{Notation}\label{notation}

We will use all the notation and set-up of \cite[Section 2]{W}, specialised to the case of the dual pair $(G,G')$, with $G'=\Sp_{2a}$ the isometry group of a $2a$-dimensional symplectic space $V'$, and $G=\O_{2k}$ the isometry group of a $2k$-dimensional orthogonal space $V$ which we assume to be the direct sum of $k$ hyperbolic planes. Note that, in the notation of the introduction, we have $G'=G_X$. 

\begin{rem}
    Strictly speaking, this specialisation is not needed until Section \ref{sec:BasicFunction}, but we do it now to fix ideas and streamline the exposition. 
\end{rem}

We also specialise to the case of the partitions (and corresponding nilpotent orbits) $\gamma'=[2a]$, $\gamma=[2a-1,1^{2k-2a+1}]$ (both even nilpotent orbits). Furthermore we take the nilpotent element $e'$ in $\mf{sp}_{2a}$ to be the `standard' regular nilpotent element.

To be precise, this means that we have a decomposition \[ V' = V'_{-(2a-1)} \oplus\cdots\oplus V'_{-1}\oplus V'_1 \oplus\cdots \oplus V'_{2a-1}, \] into one-dimensional subspaces, such that $V'_{-(2a-1)} \oplus\cdots\oplus V'_{-1}$ and $V'_1 \oplus\cdots \oplus V'_{2a-1}$ are isotropic subspaces, and there are basis vectors $v'_i\in V'_i$, with $\langle v'_i, v'_{-i} \rangle =1$ for all $i>0$. Then $e'\in \End(V')$ sends $v'_i\mapsto v'_{i+2}$ for $i=-(2a-1),\dots,2a-3$, and $v'_{2a-1}\mapsto 0$. 

We have the similar decomposition \[ V = V_{-(2a-2)}\oplus\cdots \oplus V_{-2}\oplus V_0\oplus V_2\oplus\cdots\oplus V_{2a-2}, \] with $\dim V_0 = 2k-2a+2$, $V_{-(2a-2)}\oplus\cdots \oplus V_{-2}$ and $V_2\oplus\cdots\oplus V_{2a-2}$ isotropic, $V_i\oplus V_{-i}$ a hyperbolic plane (in the obvious way) for all $i\ne 0$, and the (split) form on $V$ restricting to a non-degenerate (split) form on $V_0$. 

Set $r=2a-2$. As in \cite[Section 2.2.2]{W}, define a flag $(\overline{V}_t)$ of subspaces of $V$ by \[ \overline{V}_t := \bigoplus_{j=-r}^t V_j, \] then we have a parabolic $P=MU$ which is the stabiliser of this flag. One has, as in \cite[Section 2.2]{W}, a unitary character $\chi_\gamma$ of $U$.

Denote, for $0\le m\le r$, \[ V_{(m)} := V_{-m}\oplus\cdots\oplus V_m, \] and $G_{(m)} = G(V_{(m)})$, the isometry group of $V_{(m)}$.

For $1\le m\le r$, denote by $P_m$ the stabiliser of $V_{-m}$ in $G_{(m)}$; it is a parabolic subgroup of $G_{(m)}$ with $P_m = M_{(m)} U_m$, where $M_{(m)} = M_m \times G_{(m-1)}$ with $M_m\cong \GL(V_{-m})$. If $U_{(m)} := U \cap G_{(m)}$, then one has \[ U_{(m)}=\big( (U_1\ltimes U_2)\ltimes \cdots \big) \ltimes U_m. \]

One may define $\chi_{\gamma,m} = \chi_\gamma |_{U_m}$, and then clearly $\chi_\gamma = \chi_{\gamma,1}\cdots \chi_{\gamma,r}$.

For each $m$, the Lie algebra $\mf{u}_m$ has a vector space decomposition
 \[ \mf{u}_{m} \cong \mf{q}_{m} \oplus \mf{z}_{m} \cong \Hom(V_{(m-1)},V_{-m}) \oplus \mf{z}_{m}.  \] 
 Here $\mf{z}_{m}$ is a Lie subalgebra which  is central in $\mf{u}_m$, and given by
   \[ \mf{z}_m = \{ z\in \Hom(V_m, V_{-m}) | z^\ast = -z\} \subseteq \g_{-2m}, \] 
   while the isomorphism $ \Hom(V_{(m-1)},V_{-m})\cong \mf{q}_{m} $ is provided by 
   \[ q\mapsto q-q^\ast,  \]
where $\ast$ denotes the adjoint or dual linear map. 

One has a corresponding short exact sequence \[ 1\rightarrow Z_m \rightarrow U_m \rightarrow Q_m \rightarrow 1.\]

One makes similar definitions for the group $G'$, except that $P'_m$ is the stabiliser of $V'_{m}$ in $G'_{(m)}$, and we have the identification (again as vector spaces) \[ \mf{q}'_{m} \cong \Hom(V'_{m},V'_{(m-1)}). \] 

\vskip 5pt

As in \cite[Section 2.4]{W}, write $\omega_{(l),(m)}$ for the Weil representation associated to the dual pair $(G_{(l)}, G'_{(m)})$; it is realised on the space of Schwarz functions on a Lagrangian $Y_{(l),(m)}$ of the symplectic vector space $\Hom(V_{(l)},V'_{(m)})$. 

We choose Lagrangian subspaces once and for all in the following way. Choose a Lagrangian $Y_{(0),(0)}$ of $\Hom(V_{(0)},V'_{(0)})$. Then for each $0\le j\le r$, set \[ Y_{(j),(j+1)} = \Hom(V_{(j)}, V'_{j+1}) \oplus Y_{(j),(j)} \] and for each $1\le j\le r$, set \[ Y_{(j),(j)} = \Hom(V_{-j}, V'_{(j)}) \oplus Y_{(j-1),(j)}. \] 

\vskip 5pt

Finally, as in \cite[Proposition 2.3]{W}, we have an element $f\in\Hom(V,V')$ facilitating the moment map transfer of nilpotent orbits, with \[ f = \bigoplus_{j=-r}^r f_j \] (here $r=2a-2$) and $f_j\in\Hom(V_j,V'_{j+1})$ for each $j$. We have that $f_0^\ast v'_{-1}$ is of norm 1 in $V_0$, and $(f_0^\ast v'_{-1})^\perp = \ker f_0$. Recall that $L$ is the isometry group of this subspace $(f_0^\ast v'_{-1})^\perp = \ker f_0$ of dimension $2k-2a+1$. 

Denote
 \[ X:=LU, \overline{\chi_\gamma}\bs G, \] 
 and recall that $M_2$ is a twisted cotangent bundle of $X$. More precisely,  $X$ represents the data of the space $LU\bs G$, together with the unitary character $\overline{\chi_\gamma}$ of $LU$ (with $L$ acting trivially);  see also \cite[Section 3.2.1]{BZSV} for more on twisted cotangent bundles.

\subsection{Measures} We take on the groups $G=G_{(2a-2)},\dots,G_{(0)},L$ (which are smooth group schemes over $\o_F$), invariant, integral volume forms which are residually non-vanishing. This defines a Haar measure on $X$ (also referred to as \textit{`Tamagawa measure'} in \cite{BZSV}), under which \[ \vol(X(\o_F)) = q^{-\dim X} |X(\kappa_F)|.  \]

This differs from the normalisation ultimately used in \cite[Section 9]{BZSV}, but we shall return to this relatively minor issue later (in Section \ref{sec:BasicFunction}). 

In general, on all spaces we will work with the measure associated to an invariant, integral volume form which is residually non-vanishing, or `Tamagawa measure'.

\section{Spectral decomposition}\label{sec:Spectral}

The aim of this section, which is the technical heart of this paper, is to show a spectral decomposition for (the inner product on) $L^2(X)=L^2(LU, \overline{\chi_\gamma}\bs G)$. 

\subsection{Smooth setting} Recall that the smooth Weil representation $\omega^\infty = \omega_{(r),(r+1)}^\infty$ (where here, $r=2a-2$) associated to the dual pair $(G',G)$, is realised on the space $\S(Y_{(r),(r+1)})$ of Schwarz functions on the Lagrangian $Y_{(r),(r+1)}$ of $\Hom(V,V')$. See \cite[Section 2.4]{W} for explicit formulas for the action on the Weil representation. 

Recall also that the main result of \cite{GZ} is the following (in the smooth setting)

\begin{thm}\label{thm:SmoothMain}
The map   \begin{align*}
    p : \omega^\infty = \S(Y_{(r),(r+1)}) &\rightarrow \S(LU, \overline{\chi_\gamma}\bs G)\\
    \Phi &\mapsto \phi:= \big(g \mapsto (g\cdot\Phi)(f)\big)
\end{align*} 
is surjective, and induces the isomorphism $(\omega^\infty)_{U',\chi_{\gamma'}} \cong \S(LU, \overline{\chi_\gamma}\bs G) $.
\end{thm}

Write $\phi=p(\Phi)$, $\phi_1=p(\Phi_1)$, etc. throughout.

\subsection{Plancherel theorems}\label{sec:PlancherelRecall} Recall now the spectral decomposition of the Weil representation $\omega$ as first explicated in \cite{Sa} and later also in \cite[Section 5]{GW2}, which may (equivalently) be formulated as a spectral decomposition of the inner product $\langle \cdot,\cdot\rangle_\omega$ on $\omega$, that is, \[ \langle \Phi_1, \Phi_2 \rangle_\omega = \int_{\widehat{G'}}\langle \theta_\sigma(\Phi_1), \theta_\sigma(\Phi_2) \rangle_{\sigma\otimes\theta(\sigma)} \mathop{d\mu_{G'}(\sigma)} \] for $\Phi_1,\Phi_2\in\omega^\infty$, where $\widehat{G'}$ denotes (as usual) the unitary dual of $G'$, $d\mu_{G'}$ the Plancherel measure of $G'$ (associated to the Haar measure on $G'$), and $\theta_\sigma$ is the canonical map \[ \theta_\sigma: \omega^\infty \rightarrow \sigma\otimes \theta_\psi(\sigma) \] obtained from the smooth theta correspondence. (Recall that we fix the additive character $\psi$, and so suppress it from all notation.)

Recall also the Whittaker-Plancherel theorem for $(U',\chi_{\gamma'})\bs G'$ (as explicated in \cite[Section 3.2]{GW2} for instance), which furnishes canonical Whittaker functionals \[ l_\sigma \in \Hom_{U'}(\sigma,\chi_{\gamma'}), \] with \[ l_\sigma \otimes \overline{l_\sigma} : v_1\otimes v_2 \mapsto \int_{U'}^\ast \overline{\chi_{\gamma'}(u')} \langle u'\cdot v_1, v_2 \rangle_\sigma du', \] where $^\ast$ here indicates that the integral should be suitably regularised as in \cite{LM}.

From Theorem \ref{thm:SmoothMain}, it is clear that the below commutative square can be completed by a map \[ \alpha_{\theta(\sigma)} : \S(LU,\overline{\chi_\gamma}\bs G) \rightarrow \theta(\sigma). \] 
The question now is whether this (family of) maps also facilitates the spectral decomposition of $X$, or equivalently the inner product $\langle \cdot,\cdot\rangle_X$. This is addressed in the next subsection. 

\[\begin{tikzcd}
& \S(Y_{(r),(r+1)}) \arrow[dl,"{p}"]  \arrow[rd, "{\theta_\sigma}"] &\\
\S(LU,\overline{\chi_\gamma}\bs G)\arrow[dr,"\alpha_{\theta(\sigma)}"] & & \sigma\otimes\theta(\sigma) \arrow[ld,"{l_\sigma}"]\\
& \theta(\sigma) & 
\end{tikzcd}\]

This is the key point of departure from the previous work of \cite{GW2} on Plancherel decomposition (à la Sakellaridis-Venkatesh). There, the spectral decomposition of $L^2(X)$ was first shown in a different manner, which does not generalise to the cases we consider here, and only afterwards compared with the above commutative diagram. Here, we show such a spectral decomposition in a more general manner which is in fact parallel to the (smooth) local and global aspects of the problem.

\subsection{Key argument} In the next Section \ref{sec:InnerProduct}, we will show that

\begin{prop}\label{prop:InnerProduct2}
For $\Phi_1,\Phi_2\in\omega^\infty$, one has \[  \langle \phi_1,\phi_2 \rangle_X = \int_{U'} \overline{\chi_{\gamma'}(u')} \langle u'\cdot \Phi_1, \Phi_2 \rangle_\omega du'. \]
\end{prop}

Therefore, assuming this proposition for the moment, we see that in order to analyse the inner product $\langle\cdot,\cdot\rangle_X$, it is enough to consider the integral on the RHS, and we do so first. 

From the above spectral decomposition of the Weil representation $\omega$, we have that \[ \int_{U'} \overline{\chi_{\gamma'}(u')} \langle u'\cdot \Phi_1, \Phi_2 \rangle_\omega du' = \int_{U'} \overline{\chi_{\gamma'}(u')} \int_{\widehat{G'}}\langle \theta_\sigma(u'\cdot \Phi_1), \theta_\sigma(\Phi_2) \rangle_{\sigma\otimes\theta(\sigma)} \mathop{d\mu_{G'}(\sigma)} \mathop{du'} \]

We would now like to justify an interchange in the order of integration.

To that end, consider first the integral \begin{align*}  &\int_{\widehat{G'}} \int_{U'}^\ast \overline{\chi_{\gamma'}(u')} \langle \theta_\sigma(u'\cdot \Phi_1), \theta_\sigma(\Phi_2) \rangle_{\sigma\otimes\theta(\sigma)} \mathop{du'} \mathop{d\mu_{G'}(\sigma)}  \\ &= \int_{\widehat{G'}} \langle l_\sigma(\theta_\sigma( \Phi_1)), l_\sigma(\theta_\sigma(\Phi_2)) \rangle_{\theta(\sigma)} \mathop{d\mu_{G'}(\sigma)}  \end{align*}

We now make two key observations:
\begin{itemize}
    \item Since $\Phi_1,\Phi_2\in\omega^\infty$ (hence fixed by a compact open subgroup), by \cite[Proposition 2.3]{LM}, the inner integral over $U'$ may (and should be interpreted as being) replaced by an integral over any compact open subgroup $U_x'$ containing a fixed compact open subgroup $U_0'$ of $U'$,  which depends only on $\Phi_1,\Phi_2$ (and not on $\sigma$).

    \item Also since $\Phi_1,\Phi_2\in\omega^\infty$ (hence fixed by a compact open subgroup), the integral over $\widehat{G'}$ is compactly supported (cf. \cite[Théorème VIII.1.2]{Wa}). 
\end{itemize}

Therefore (the integral exists and) we may interchange the order of integration; in other words, we obtain 

\begin{prop}\label{prop:SpectralDecomp}
We have 
    \[ \int_{U'} \overline{\chi_{\gamma'}(u')} \langle u'\cdot \Phi_1, \Phi_2 \rangle_\omega \mathop{du'} = \int_{\widehat{G'}} \langle l_\sigma(\theta_\sigma( \Phi_1)), l_\sigma(\theta_\sigma(\Phi_2)) \rangle_{\theta(\sigma)} \mathop{d\mu_{G'}(\sigma)}. \]
\end{prop}

\begin{rem}
    The interchange in order of integration could also be justified by a similar argument to that in the proof of \cite[Theorem 6.3.4]{SV}.
\end{rem}

Combining Propositions \ref{prop:InnerProduct2} and \ref{prop:SpectralDecomp}, we obtain: 

\begin{thm}\label{thm:MainSpectral} 
We have 
    \begin{align*} \langle \phi_1,\phi_2 \rangle_X &= \int_{\widehat{G'}} \langle l_\sigma(\theta_\sigma(\Phi_1)), l_\sigma(\theta_\sigma(\Phi_2)) \rangle_{\theta(\sigma)} \mathop{d\mu_{G'}(\sigma)} \\
    &= \int_{\widehat{G'}} \langle \alpha_{\theta(\sigma)}(\phi_1), \alpha_{\theta(\sigma)}(\phi_2) \rangle_{\theta(\sigma)} \mathop{d\mu_{G'}(\sigma)} \end{align*}
    for all $\phi_1,\phi_2\in\S(X)$.
\end{thm}

In other words, we have obtained a spectral decomposition of the inner product on $X=LU,\overline{\chi_\gamma}\bs G$, which we now see is facilitated by the family of maps obtained from the commutative square in the previous subsection \ref{sec:PlancherelRecall} (together with the Plancherel measure of $G'$).  

We may also state this as

\begin{thm}\label{thm:MainSpectral2}
 We have 
 \[ L^2(X) \cong \int_{\widehat{G'}} \theta_\psi (\sigma) \cdot \dim \Hom_{U'}( \sigma, \chi_{\gamma'}) \mathop{d\mu_{G'}(\sigma)}, \] 
so that the spectral measure on $L^2(X)$ is given by the pushforward, via theta lifting, of the (Whittaker-)Plancherel measure on $G'$, which is supported on the tempered, $\psi$-generic irreducible representations of $G'$. 
    
 Furthermore, the resulting family of projections  \[ \alpha_{\theta(\sigma)} : \S(X) \rightarrow \theta(\sigma) \] is that defined by the commutative diagram of Section \ref{sec:PlancherelRecall}.  
\end{thm}

\begin{rem}
    Although we have worked exclusively with the non-Archimedean case (as is also done in \cite{BZSV}), it is to be expected that the analogous results concerning spectral decomposition are true as well in the Archimedean case, albeit with greater analytic difficulties to be considered. 
\end{rem}

\section{Inner product}\label{sec:InnerProduct}

The aim of this section (which may be thought of as an appendix to Section \ref{sec:Spectral}) is to show the following (earlier stated as Proposition \ref{prop:InnerProduct2}):

\begin{prop}\label{prop:InnerProduct}
For $\Phi\in\omega^\infty$, one has \[ \int_{U'} \overline{\chi_{\gamma'}(u')} \langle u'\cdot \Phi_1, \Phi_2 \rangle_\omega du' = \langle \phi_1,\phi_2 \rangle_X. \]
\end{prop}

\begin{proof}
First, we remark that the integral over $U'$ is not necessarily convergent in general. But in this case, the matrix coefficient of the Weil representation lies in the Harish-Chandra Schwarz space of $G'$, and so the integral over $U'$ does at least converge. However, for our purposes, we would like to say something stronger. 

From the first key observation in the preceding section (which is by \cite[Proposition 2.3]{LM}), the integral in fact \textit{stabilises}, in the sense that there is a compact open subgroup $U_0'$ of $U'$ (depending only on $\Phi_1,\Phi_2$) such that for all compact open subgroups $U_x'$ containing $U_0'$, the integral over $U_x'$ exists and takes the same value for all $U_x'$.

\begin{rem}
    We remark that this (and the argument in the preceding section) is in essence similar to the argument used in the proof of \cite[Proposition 2.11]{LM}.
\end{rem}

Therefore we may replace the integral over $U'$ with an integral over $U_x'$ for any such (sufficiently large) compact open subgroup $U_x'$. In what follows we do not notate $U_x'$ but simply understand the integral over $U'$ to be stabilised in the sense just explained. 

Furthermore, note that by definition \[ \langle u'\cdot \Phi_1, \Phi_2 \rangle_\omega = \int_{Y_{(r),(r+1)}} (u'\cdot \Phi_1)(y) \overline{\Phi_2(y)} dy.  \] Now since $\Phi_2$ is compactly supported, we may freely interchange the integrals over $U_x'$ and $Y_{(r),(r+1)}$ throughout the computation. 

\vskip 5pt

The computation is formally analogous to that carried out in \cite{W} (and so we do not repeat some of the details which are formally similar). The idea is to compute the integral over $U^{\prime}$ by inductively, or iteratively, computing the integrals first over $U^{\prime}_{r+1}$, and so on, down to $U^{\prime}_{1}$. (Recall again here $r=2a-2$.)

Therefore, we first compute 
\begin{equation}
\label{KeyInductive} 
\int_{U'_{r+1}} \overline{\chi_{\gamma'}(u')} \langle u'\cdot \Phi_1, \Phi_2 \rangle_\omega du',
\end{equation}

Recall that we have a short exact sequence  \[ 1\rightarrow Z^{\prime}_{r+1} \rightarrow U^{\prime}_{r+1} \rightarrow Q^{\prime}_{r+1} \rightarrow 1,\] with $\chi_{\gamma',r+1}$ trivial on $Z^{\prime}_{r+1}$ (for $r\ge 1$). 

So we may write \begin{align*} (\ref{KeyInductive}) &=  \int_{Q^{\prime}_{r+1}} \overline{\chi_{\gamma', r+1}(q')} \int_{Z^{\prime}_{r+1}} \int_{Y_{(r),(r+1)}} (z'q'\cdot \Phi_1)(y) \overline{\Phi_2(y)} \mathop{dy} \mathop{dz'} \mathop{dq'} 
\\ &=  \int_{Q^{\prime}_{r+1}} \overline{\chi_{\gamma', r+1}(q')} \int_{Z^{\prime}_{r+1}} \\ & \int_{\Hom(V_{(r)}, V'_{r+1}) \bs \{0\}} \int_{Y_{(r),(r)}} (z'q'\cdot \Phi_1)(F_r)(y_{(r),(r)}) \overline{\Phi_2(F_r)(y_{(r),(r)})}  \mathop{dy_{(r),(r)}}\mathop{dF_r} \mathop{dz'} \mathop{dq'}
\end{align*}

\begin{rem}
Note that $\dim V'_{r+1}=1$, so that the non-zero elements in $\Hom(V_{(r)}, V'_{r+1})$ are precisely those with maximal rank. 
\end{rem}

We first consider the integral over $Z'_{r+1}$. From the action of $Z'_{r+1}$, we have \[ (z'q'\cdot \Phi_1)(F_r)(y_{(r),(r)}) = \psi(\Tr(z'F_rF_r^\ast /2))(q'\cdot \Phi_1)(F_r)(y_{(r),(r)}). \] 

Now the map \begin{align*} q_r : \Hom(V_{(r)}, V'_{r+1}) \bs \{0\} &\rightarrow F \\
F_r &\mapsto F_rF_r^\ast \end{align*}
is a submersion (away from the origin of $\Hom(V_{(r)}, V'_{r+1})$). 

Therefore the integration over $\Hom(V_{(r)}, V'_{r+1}) \bs \{0\}$ may be performed first by integrating over the fibers of $q_r$, followed by integration along the base, and by a completely analogous argument to that in \cite[p.161-162]{GW2}, one has \begin{align*} (\ref{KeyInductive}) &=  \int_{Q^{\prime}_{r+1}} \overline{\chi_{\gamma', r+1}(q')} \\ & \int_{\substack{F_r\in \Hom(V_{(r)}, V'_{r+1}) \bs \{0\}\\ F_rF_r^\ast = 0}} \int_{Y_{(r),(r)}} (q'\cdot \Phi_1)(F_r)(y_{(r),(r)}) \overline{\Phi_2(F_r)(y_{(r),(r)})}  \mathop{dy_{(r),(r)}}\mathop{dF_r}  \mathop{dq'}
\end{align*}

For the sake of completeness, we reproduce the argument here. Observe that it is no harm to first restrict to those $\Phi_2$ which are (locally constant and) compactly supported away from the origin of $\Hom(V_{(r)}, V'_{r+1})$, and then deduce the same for general $\Phi_2$ by a density argument.

By the observation at the start of the entire proof, we replace the integral over $Z'_{r+1}$ with an integral over a sufficiently large compact open subgroup of $Z'_{r+1}$, which we may identify with $\varpi_F^{-n}\o_F$ for sufficiently large $n$. 

Now if $F_rF_r^\ast\not\in \varpi_F^n \o_F$, then $\psi(\Tr(z'F_rF_r^\ast /2))$ is a non-trivial character (in $z'$), hence the entire integral is zero. On the other hand, if $F_rF_r^\ast \in \varpi_F^n \o_F$, then the integral over $z'$ simply gives the volume of $\varpi_F^{-n}\o_F$, which is $q^n$. 

Thus \begin{align*} (\ref{KeyInductive}) &=  \int_{Q^{\prime}_{r+1}} \overline{\chi_{\gamma', r+1}(q')}  \\ & q^n \int_{q_r^{-1}(\varpi_F^n\o_F)} \int_{Y_{(r),(r)}} (q'\cdot \Phi_1)(F_r)(y_{(r),(r)}) \overline{\Phi_2(F_r)(y_{(r),(r)})}  \mathop{dy_{(r),(r)}}\mathop{dF_r} \mathop{dq'}
\end{align*}

But the integral over fibers of $q_r$ is a locally constant function on the base. Therefore we have for all sufficiently large $n$, \begin{align*} (\ref{KeyInductive}) &=  \int_{Q^{\prime}_{r+1}} \overline{\chi_{\gamma', r+1}(q')}  \\ & q^n \vol(\varpi_F^n\o_F) \int_{q_r^{-1}(0)} \int_{Y_{(r),(r)}} (q'\cdot \Phi_1)(F_r)(y_{(r),(r)}) \overline{\Phi_2(F_r)(y_{(r),(r)})}  \mathop{dy_{(r),(r)}}\mathop{dF_r} \mathop{dq'},
\end{align*}
and from $\vol(\varpi_F^n\o_F) = q^{-n}$ we get the desired. 

\vskip 5pt

Now the $F_r\in \Hom(V_{(r)}, V'_{r+1}) \bs \{0\}$ with $F_r F_r^\ast=0$ form a single orbit under the $G=G_{(r)}$-action, and in fact we already have a representative $f_r$ in this orbit coming from the moment map transfer. (This $f_r$ is trivial when restricted to $V_{-r}\oplus\cdots\oplus V_{r-1}$, and restricts to an isomorphism $V_r \xrightarrow{\sim} V'_{r+1}$.) The stabiliser of $f_r$ in $G_{(r)}$ is $G_{(r-1)}U_r$. 

Therefore, considering the action of $g\in G_{(r)}$, we have \begin{align*} (\ref{KeyInductive}) &=  \int_{Q^{\prime}_{r+1}} \overline{\chi_{\gamma', r+1}(q')} \\ & \int_{G_{(r-1)}U_r \bs G_{(r)}} \int_{Y_{(r),(r)}} \big(\omega_{(r),(r)}(g_r^{-1})(g_r q'\cdot \Phi_1)(f_r)\big)(y_{(r),(r)}) \\ &\cdot \overline{\big(\omega_{(r),(r)}(g_r^{-1})(g_r\cdot\Phi_2)(f_r)\big)(y_{(r),(r)})}  \mathop{dy_{(r),(r)}}\mathop{dg_r}  \mathop{dq'} \\
 &=  \int_{Q^{\prime}_{r+1}} \overline{\chi_{\gamma', r+1}(q')} \\ & \int_{G_{(r-1)}U_r \bs G_{(r)}} \int_{Y_{(r),(r)}} (g_r q'\cdot \Phi_1)(f_r)(y_{(r),(r)}) \overline{(g_r\cdot\Phi_2)(f_r)(y_{(r),(r)})}  \mathop{dy_{(r),(r)}}\mathop{dg_r}  \mathop{dq'}
\end{align*}
where here we have used the unitarity of $\omega_{(r),(r)}$. 

\vskip 5pt

The integral over $Q^{\prime}_{r+1}$ is handled in exactly the same way (and is in fact simpler, in some sense), by considering the action of $Q^{\prime}_{r+1}$ and again employing a completely analogous argument to that in \cite[p.161-162]{GW2}. One obtains \begin{align*} (\ref{KeyInductive}) &= \int_{G_{(r-1)}U_r \bs G_{(r)}} \int_{Y_{(r-1),(r)}} (g_r \cdot \Phi_1)(f_r)(f_{-r})(y_{(r-1),(r)}) \\ & \cdot \overline{(g_r\cdot\Phi_2)(f_r)(f_{-r})(y_{(r-1),(r)})}  \mathop{dy_{(r-1),(r)}}\mathop{dg_r} 
\end{align*}
(Here, we simply need to perform a pointwise integration over $\Hom(V_{-r},V'_{(r)})$, rather than fiberwise integration using a submersion $q_r$, which explains why the situation here is simpler.)

\vskip 5pt

Now one again proceeds iteratively to compute the integral over $U^{\prime}_{r}, \dots, U^{\prime}_{1}$, as in \cite{W}, and the computation proceeds in completely the same way as what we have just done for $U^{\prime}_{r+1}$.

At the `base case' of $U^{\prime}_{1}$, there are certain differences (which in fact simplify the computation). 
First, $Q^{\prime}_1$ is trivial, hence we need only consider the integral over $Z^{\prime}_1$; now $\chi_{\gamma',1}$ is no longer trivial on $Z^{\prime}_1$, and so one obtains an integral over the $F_0$ with \[ F_0 F_0^\ast = e' \in \Hom(V'_{-1},V'_1) \] (instead of $F_0 F_0^\ast = 0$). Then $G_{(0)}$ acts transitively on the set of such $F_0$, and we have the representative $f_0$ with stabiliser $L$.

In conclusion, we therefore obtain \begin{align*} \int_{U'} \overline{\chi_{\gamma'}(u')} \langle u'\cdot \Phi_1, \Phi_2 \rangle_\omega du' &= \int_{LU \bs G_{(r)}} (g \cdot \Phi_1)(f) \cdot \overline{(g\cdot\Phi_2)(f)}  \mathop{dg},
\end{align*}
which is as desired. 

\end{proof}

\section{Plancherel density of basic function}\label{sec:BasicFunction}

We would now like to compute the Plancherel density of the `basic function' $\phi_0$ on $X$, that is, the characteristic function of $X(\o_F)$. 

By a similar argument as in \cite[Section 4.7]{W}, note that $\phi_0=p(\Phi_0)$, where $\Phi_0\in \omega^\infty$ is the characteristic function of $Y_{(r),(r+1)}(\o_F)$. 

Therefore, by Theorem \ref{thm:MainSpectral}, what is pertinent is the computation of \[ \langle l_\sigma(\theta_\sigma(\Phi_0)), l_\sigma(\theta_\sigma(\Phi_0)) \rangle_{\theta(\sigma)} \] for (tempered, $\psi$-generic) $G'(\o_F)$-unramified $\sigma\in\widehat{G'}$. 

For each such $\sigma$, fix a $G'(\o_F)$-invariant vector $v_0 \in \sigma$ with $\langle v_0,v_0\rangle_\sigma = 1$ and a $G(\o_F)$-invariant vector $w_0\in \theta(\sigma)$ with $\langle w_0,w_0\rangle_{\theta(\sigma)} = 1$. Suppose $\theta_\sigma(\Phi_0) = c_\sigma v_0\otimes w_0$. Then \[ \langle l_\sigma(\theta_\sigma(\Phi_0)), l_\sigma(\theta_\sigma(\Phi_0)) \rangle_{\theta(\sigma)} = |c_\sigma|^2 \cdot |l_\sigma(v_0)|^2.  \]

Both these quantities have been computed in \cite[Proposition 3]{LR} and \cite[Proposition 2.14]{LM} respectively. In particular, we have \[ |c_\sigma|^2 = Z(\Phi_0,\Phi_0,v_0,v_0) = \vol(G'(\o_F)) \frac{L(k-a,\sigma,\std)}{\zeta_F(k)\cdot \prod_{j=1}^a \zeta_F(2k-2j)} \] where $Z$ denotes the local doubling zeta integral (note here that we are in the second case of \cite[Remark 3]{LR}), and \[ |l_\sigma(v_0)|^2 = \frac{\Delta_{G'}(1)}{L(1,\sigma,\Ad)}, \] where $\Delta_{G'}$ is the $L$-factor of the dual to the motive associated to $G'$ (introduced by Gross in \cite{Gr}). 

Observe that  $\vol(G'(\o_F)) = \Delta_{G'}(1)^{-1}$ \cite[Prop. 4.7]{Gr}, so that \[ \langle l_\sigma(\theta_\sigma(\Phi_0)), l_\sigma(\theta_\sigma(\Phi_0)) \rangle_{\theta(\sigma)} = \frac{L(k-a,\sigma,\std)}{L(1,\sigma,\Ad)\cdot \zeta_F(k)\cdot \prod_{j=1}^a \zeta_F(2k-2j)}  . \]

Now, recall that the choice of normalisation of measure in \cite{BZSV} is that of the `Tamagawa measure' on $X$ divided by the `Tamagawa volume' of $G(\o_F)$, which in this case (considering the exponents of $G=\O_{2k}$) is 
\[ \Delta_G(1)^{-1} = (\zeta_F(k) \cdot \prod_{j=1}^{k-1} \zeta_F(2k-2j))^{-1}; \] 
see \cite[Section 2.4]{LM}. In conclusion, we see that one obtains:

\begin{prop}\label{prop:MainComputation}
We have
    \begin{equation}\label{MainComputation} \langle \phi_0,\phi_0 \rangle'_X = \int_{\widehat{G'}_{\text{unramified}}} \frac{L(k-a,\sigma,\std)\cdot \zeta_F(2k-2a-2)\cdots\zeta_F(4)\zeta_F(2)}{L(1,\sigma,\Ad)} \mathop{d\mu_{G'}(\sigma)}, \end{equation} Here, $'$ indicates the renormalised (as in \cite{BZSV}) measure on $X$. 
\end{prop}

It remains now to verify that this matches the expectations for the hyperspherical dual $M_1$. 

In the notation of \cite{BZSV}, for the hyperspherical dual $M_1$, we have $(S=)S_X=0$, and a routine computation (using the fact that one has $\ad = \wedge^2 \std$ as representations of $\mf{so}_{2k}$) shows that 
\begin{lem}\label{lem:SliceComputation}
We have
    \[ V_X=\mf{so}_{2a+1}^\perp \cap \mf{so}_{2k}^e = \std[2k-2a] \oplus \triv[4k-4a-4,\cdots,8,4] \] as $\O_{2a+1}$-representations.
\end{lem}

(Recall that the grading on $\mf{so}_{2a+1}^\perp \cap \mf{so}_{2k}^e$ is given by adding 2 to the weights of the adjoint $\mfsl_2$-action on $\mf{so}_{2k}$.)

\begin{rem}
    Recall that, as in the introduction, the significance of the vector space $V_X$ is that one has an isomorphism $M_1 \cong V_X \times^{G'^\vee} G^\vee$.
\end{rem}

Now note that the point(s) of evaluation of the $L$-factor associated to $V_X$ are (expected to be) given by $\frac{1}{2}$ times the grading on $V_X$. Proposition \ref{prop:MainComputation} taken together with Lemma \ref{lem:SliceComputation} therefore already shows that the obtained Plancherel density is in remarkable agreement with the general conjecture of \cite{BZSV}. 

\vskip 5pt

In order to verify that one has a precise agreement with the explicit numerical conjecture in the form given in \cite[Proposition 9.2.1]{BZSV}, it remains to make two further observations. 

First, the fact that the Plancherel measure is given as the pushforward under multiplication by $q^{-\rho_{L(X)}}$ is a direct consequence of \cite[Theorem 8.1(b)]{Ga}. Of course, this is simply a reflection of the expected characterisation of the spectral decomposition as functorial lifting along the map $\O_{2a+1}\times \SL_2\rightarrow \O_{2k}$. Here, $\rho_{L(X)}$ is $\frac{1}{2}$ of the (co)weight for $\O_{2k}$ corresponding to the map $\mfsl_2\rightarrow \mf{so}_{2k}$. 

Second, the Plancherel formula of \cite[Proposition 9.2.1]{BZSV} is given in terms of the Haar measure $dt$ on the maximal compact subgroup $A_X^{\vee (1)}$ of the torus $A_X^\vee$ in $G_X^\vee = \O_{2a+1}$, via the Satake isomorphism. (Recall that we have $G'=G_X$.) 

However, cf. also the group case of the same conjecture (which has essentially been shown by Macdonald \cite{Ma} and later also by Casselman \cite{Ca}), we may express the Plancherel measure $d\mu_{G'}(\sigma)$ (on unramified representations) as \[ d\mu_{G'}(\sigma) = \frac{1}{|W_X|} \frac{L(1,\sigma,\Ad)}{L(0,\sigma,\g_X^\vee/\mf{a}_X^\vee)} dt. \]

\begin{rem}
The astute reader would no doubt notice that there is a potential issue with normalisations to be addressed; namely, the group case of the same conjecture entails a renormalisation which gives $G'(\o_F)$ a volume equal to the \textit{reciprocal} of the `Tamagawa volume' $\vol(G'(\o_F))$ that we have given it in this paper. In other words we should have \[ \frac{1}{\vol(G'(\o_F))} = \frac{1}{|W_X|}\int_{A_X^{\vee (1)}}\frac{L(1,\sigma,\Ad)}{L(0,\sigma,\g_X^\vee/\mf{a}_X^\vee)} dt.  \]

However the (Harish-Chandra-)Plancherel formula for $G'$ (with its Tamagawa measure) gives \[  \vol(G'(\o_F)) =  \int_{\widehat{G'}_{\text{unramified}}} \vol(G'(\o_F))^2 d\mu_{G'}(\sigma), \] and comparing the two shows that we have the correct normalisation for the Tamagawa measure on $G'$. 
\end{rem}

The above two observations, taken together with the main Proposition \ref{prop:MainComputation}, show

\begin{thm}\label{thm:MainThm}
The numerical conjecture for the Plancherel density \cite[Proposition 9.2.1]{BZSV} is true for the hyperspherical dual pair $M_1$ and $M_2$. 
\end{thm}

To be precise, what we have shown is:

\begin{thm}\label{thm:MainThm2}
    The Plancherel density of the basic function $\phi_0$, which is a measure on $\widehat{G}_{\text{unramified}}$, and hence (via Satake) identified with a measure on $A^\vee/W$, where $A^\vee$ is the maximal torus in $\O_{2k}$ and $W$ its Weyl group, is given as the pushforward via \begin{align*}
        A_X^{\vee (1)} &\rightarrow A^\vee \\
        (-) &\mapsto (-)\cdot q^{-\rho_{L(X)}}
    \end{align*}
    of the measure on $A_X^{\vee (1)}$ given by \[ \frac{1}{|W_X|} \frac{ \det (I - t|_{\g_X^\vee/\mf{a}_X^\vee}) }{\det(I-(t,q^{-1/2})|_{V_X}) } dt. \]
\end{thm}

\vskip 10pt

\section{Global period}\label{sec:GlobalPeriod}

The global period conjecture, at least as stated in the setting of \cite[Section 14.6.4]{BZSV}), can now be verified in exactly the same way as in \cite[Section 12.14]{GW2}, conditional on the corresponding conjecture for the Whittaker period on general symplectic groups $G'=\Sp_{2a}$ \cite[Conjecture 3.3]{LM}.

   Let us be more precise.  Fix a number field $k$ with ring of adeles $\A$ and a non-trivial unitary additive character $\psi = \otimes_v \psi_v : k\bs \A \rightarrow \C^\times$, and for any algebraic group $G$ over $k$, denote as usual $[G]:= G(k)\bs G(\A)$. For $f$ a cuspidal automorphic form on $G$, denote the period integral 
   \[ P_{\gamma}(f) = \int_{[L]} \int_{[U]} f(ul) \cdot \overline{ \chi_{\gamma}(u)  } \mathop{du} \mathop{dl}, \] 
   taking as usual Tamagawa measure on all adelic groups. Here $\gamma = [2a-1, 1^{2k-2a+1} ]$ is the partition we have been considering in the local setting, which is the image of  the partition $\gamma' = [2a]$ under the moment map correspondence for unipotent orbits of $(G, G')$.

We first note the following result from  \cite[Thm. 4.6 and Thm. 4.7]{W}.

\begin{thm}
Let $\Pi$ be an irreducible  cuspidal automorphic representation of $G = \O_{2k}$ and consider its  global theta lift $\Theta_{\psi}(\Pi)$  to $G' = \Sp_{2a}$. 
Then $\Theta_{\psi}(\Pi)$ is $\psi$-generic if and only if the automorphic period $P_{\gamma}$ is nonzero on $\Pi$. 
\end{thm}

Since we are interested in studying the automorphic period $P_{\gamma}$, let us assume the conclusion of the above theorem holds. 
 We shall further assume that $\Theta_{\psi}(\Pi) =: \Sigma$ is cuspidal. In this case,  being a globally $\psi$-generic cuspidal representation of $G'$, $\Sigma$ has a generic  L-parameter $\phi$  and  satisfies $\Theta_{\psi}(\Sigma) = \Pi$ (by \cite[Pg. 231, Th\'eorem\`e]{Mo}). Moreover, by results on the  theta correspondence of unramified representations (see \cite[Thm. 8.1 and Sect. 8.2]{Ga} for example),  the cuspidal representation $\Pi =\Theta_{\psi}(\Sigma)$ has A-parameter of the form $\phi \oplus S_{2k-2a-1}$, where $S_{2k-2a-1}$ denotes the $(2k-2a-1)$-dimensional irreducible representation of the Arthur $\SL_2$. Formulated in terms of the hypothetical Langlands group $\mathcal{L}_k$, the A-parameter of $\Pi$ thus  factors as: 
\begin{equation} \label{E:Arthur}
\begin{CD}
 \mathcal{L}_k \times \SL_2 @>\phi \times S_{2k-2a-1}>>\SO_{2a+1}(\C)    \times \SO_{2k-2a-1}(\C) @>>> \O_{2k}(\C).
 \end{CD} \end{equation}
In particular, we are putting ourselves in the setting of \cite[Section 14.6.4]{BZSV}.

Here is our first global result:

\begin{thm}\label{thm:GlobalPeriod}
Assume \cite[Conjecture 3.3]{LM} for $G'$, and that one has local multiplicity one for $LU, \overline{\chi_\gamma}\bs G$ (this is the less serious hypothesis). 

Let $\Pi$ be an irreducible  cuspidal automorphic representation of $G = \O_{2k}$ whose global theta lift $\Sigma := \Theta_{\psi}(\Pi)$ on $G' = \Sp_{2a}$ is cuspidal and globally $\psi$-generic, with corresponding generic L-parameter $\phi$ and associated global component group $S_{\phi}$. 
Let $T$ be a finite set of places (including the Archimedean ones) outside which all data are unramified. Then the restriction of the functional $P_\gamma$ to $\Pi$ factorises as 
\[ |P_\gamma(f)|^2 = \frac{1}{|S_\phi|} \frac{\Delta_G^T(1)}{\Delta_L^T(1)^2} \frac{\prod_i L^T(i/2, \Sigma, V_X^{(i)})}{L^T(1, \Sigma, \Ad) } \prod_{v\in T} |l'_{\Pi_v} (f_v)|^2 \]
for $f=\otimes_v f_v \in \otimes_v \Pi_v \cong \otimes_v \theta_{\psi_v}(\Sigma_v)$ and where $l'_{\Pi_v}$ are the (normalised as in \cite{GW2}, so that $l'_{\Pi_v} (w_{0,v})=1$ for almost all $v$) canonical functionals obtained from the Plancherel decomposition in Theorem \ref{thm:MainSpectral} and Theorem \ref{thm:MainSpectral2} .  
\end{thm}

\begin{proof}
    As noted above, the proof is completely analogous to that in \cite{GW2}, since the local spectral decomposition has been shown in this paper, and the necessary global computation \cite[Proposition 12.2]{GW2} has been performed in \cite{W}. We will give a sketch (of the key argument) here, omitting the other formal details which can be found in \cite{GW2}. 

Define, in the global setting, the global theta lift
    \begin{align*}
        B^{\text{Aut}}_\Pi : \omega \otimes \overline{\Pi} &\rightarrow \Sigma \\
        \Phi \otimes \overline{f} &\mapsto \Big( g' \mapsto \int_{[G]} \theta(\Phi)(g'g)\overline{f(g)} dg \Big).
    \end{align*}
The image of $B^{\text{Aut}}_\Pi$ is $\Sigma$, by multiplicity one for $G'=\Sp_{2a}$ (at least for tempered A-packets, since $\Sigma$ is globally generic) \cite{Ar}.

 Now the main result of \cite{W} (cf. Section 4.6 there) can be stated as:
     \begin{equation}\label{eq:Global}  P_{\gamma'} (B^{\text{Aut}}_\Pi(\Phi,f)) = \langle p(\Phi), \beta^{\text{Aut}}_\Pi(f) \rangle_{X(\A)}, \end{equation}
where $p$ is defined as in Section \ref{sec:Spectral}, and here
 \begin{align*}
        \beta^{\text{Aut}}_\Pi : \Pi &\rightarrow C^\infty(X(\A)) \\
        f &\mapsto \Big( g \mapsto P_\gamma(g\cdot f)  \Big).
    \end{align*}
On the other hand, the spectral decomposition shown in this paper provides a local analog of this result.  Indeed, we have, in the local setting, for any $w\in \theta(\sigma)$, \begin{align}\label{eqn:LocalPeriodRelation}
        \langle \alpha_{\theta(\sigma)}(p(\Phi)), w\rangle_{\theta(\sigma)} &=  \langle l_\sigma(\theta_\sigma(\Phi)), w\rangle_{\theta(\sigma)} \\
        &= l_\sigma (B_{\theta(\sigma)}(\Phi, w)) \nonumber
    \end{align}
    where here \begin{align*}
        B_{\theta(\sigma)} : \omega \otimes \overline{\theta(\sigma)} &\rightarrow \sigma \\
        \Phi\otimes w &\mapsto \langle \theta_\sigma(\Phi), w\rangle_{\theta(\sigma)}
    \end{align*}
    is the local version of the global theta lifting $B^{\text{Aut}}_\Pi$.

    We have the dual or adjoint \[ \beta_{\theta(\sigma)} : \theta(\sigma) \rightarrow C^\infty(X) \] to the family of maps $\alpha_{\theta(\sigma)}$ which facilitates the spectral decomposition of $X$, and 
    \begin{equation}\label{eq:Local}  l_\sigma (B_{\theta(\sigma)}(\Phi, w)) = \langle p(\Phi), \beta_{\theta(\sigma)}(w) \rangle_{X}. \end{equation}
Note that $\beta_{\theta(\sigma)}$ is the local version of $\beta^{\text{Aut}}_\Pi$, and that $l_{\theta(\sigma)} = \text{ev}_{\id_G} \circ \beta_{\theta(\sigma)}$.

    To summarise, we have two main results (\ref{eq:Global}) and (\ref{eq:Local}), with (\ref{eq:Local}) the local analog of the global (\ref{eq:Global}). 

    By multiplicity one, we have factorisations \[  P_{\gamma'} = c(\Sigma) \cdot \prod_v^\ast l_{\Sigma_v}, \] \[ B_\Pi^{\text{Aut}} = b(\Sigma) \cdot \prod_v^\ast B_{\Pi_v}, \] \[  P_{\gamma} = c(\Theta(\Sigma)) \cdot \prod_v^\ast l_{\Pi_v}, \] and hence also \[  \beta_{\Pi}^{\text{Aut}} = c(\Theta(\Sigma)) \cdot \prod_v^\ast \beta_{\Pi_v}, \] for constants $c(\Sigma), b(\Sigma), c(\Theta(\Sigma))$, and where $^\ast$ denotes suitable (standard) regularisation as in \cite[Section 12.12]{GW2}. 

  Comparing (\ref{eq:Global}) and (\ref{eq:Local}), it is clear that we should have 
  \[ c(\Sigma)\cdot b(\Sigma) = \overline{c(\Theta(\Sigma))}. \]
  Now \cite[Conjecture 3.3]{LM} gives $|c(\Sigma)|^2 = \frac{1}{|S_\phi|}$, and as in \cite[Proposition 12.4]{GW2} (by the Rallis inner product formula), we have $|b(\Sigma)|=1$.  Therefore $|c(\Theta(\Sigma))|^2 = \frac{1}{|S_\phi|}$, as desired. 

    Finally, we note that, in the local unramified computation, \begin{equation}\label{LocalUnramifiedComputation} |l_{\theta(\sigma)}(w_0)|^2 =   \vol(X(\o_F))^{-2} \cdot \langle \alpha_{\theta(\sigma)}(\phi_0) , \alpha_{\theta(\sigma)}(\phi_0) \rangle_{\theta(\sigma)}, \end{equation} with the latter quantity already computed in Section \ref{sec:BasicFunction} (but here, we do not do the renormalisation of \cite{BZSV}). This explains the extra factor of $\frac{\Delta_G^T(1)}{\Delta_L^T(1)^2}$ in the final expression.

\end{proof}

\section{Degenerate Whittaker periods}\label{degenerate_whittaker}

In this section, we show how the key conjectures of \cite{MWZ} concerning degenerate Whittaker periods, which arise from a relative trace formula approach to the global period conjectures of \cite{BZSV}, can be addressed efficiently using the framework of theta correspondence set up in this paper. In particular, we will establish \cite[Conjectures 1.4 and 1.8]{MWZ} in our setting.  

\vskip 5pt

We continue to work in the global setting of the previous section and in particular, the setting of Theorem \ref{thm:GlobalPeriod}. There, we considered cuspidal automorphic representations $\Pi$ of $G = \O_{2k}$ whose A-parameters factor as in (\ref{E:Arthur}). The restriction of such A-parameters to the Arthur $\SL_2$ gives rise to a nilpotent orbit of $G^{\vee} = \O_{2k}(\C)$ with associated partition
 \[ \gamma^\vee=[2k-2a-1, 1^{2a+1}]. \]
 The Barbasch-Vogan dual nilpotent orbit \cite{BV} to $\gamma^{\vee}$ corresponds to the partition
  \[ \gamma^\vee_{BV}=[2a+1, 1^{2k-2a-1}]. \] 
 which in turn determines a class of degenerate Whittaker periods on $G = \O_{2k}$.  These are the degenerate Whittaker period addressed by \cite[Conjectures 1.4 and 1.8]{MWZ}.
\vskip 5pt

\subsection{Preliminaries and set-up} 
 Let us briefly recall the notation and set-up of Section \ref{notation}.   In what follows, the subscript ${}_{BV}$ denotes an object associated to $\gamma^\vee_{BV}$ (in contrast to $\gamma$); note that they are nilpotent orbits of the same group $G=\O_{2k}$. 

Similar to the set-up of Section \ref{notation} but now for $\gamma^\vee_{BV}$ instead of $\gamma$, we have a new decomposition  
 \[ V = V_{-2a,BV}\oplus\cdots \oplus V_{-2,BV}\oplus V_{0,BV}\oplus V_{2,BV}\oplus\cdots\oplus V_{2a,BV}, \] with $\dim V_{0,BV} = 2k-2a$, $V_{-2a,BV}\oplus\cdots \oplus V_{-2,BV}$ and $V_{2,BV}\oplus\cdots\oplus V_{2a,BV}$ isotropic, $V_{i,BV}\oplus V_{-i,BV}$ a hyperbolic plane (in the obvious way) for all $i\ne 0$, and the (split) form on $V$ restricting to a non-degenerate (split) form on $V_{0,BV}$. 

Again similarly, set $r=2a$, and as in \cite[Section 2.2.2]{W}, define a flag $(\overline{V}_{t,BV})$ of subspaces of $V$ by \[ \overline{V}_{t,BV} := \bigoplus_{j=-r}^t V_{j,BV}, \] then we have a parabolic $P_{BV}=M_{BV}U_{BV}$ which is the stabiliser of this flag. One has, as in \cite[Section 2.2]{W}, a unitary character $\chi_{\gamma^\vee_{BV}}$ of $U_{BV}$.

Now we observe the key fact: that $\gamma^\vee_{BV}$ and the regular nilpotent orbit $\gamma'$ of Section \ref{notation} are \textit{in fact in moment map correspondence with each other (just in the opposite direction from that of $\gamma$ and $\gamma'$)}! 

Indeed, as in \cite[Proposition 2.3]{W}, we have an element $f_{BV}\in\Hom(V',V)$ facilitating the moment map transfer of nilpotent orbits, with \[ f_{BV} = \bigoplus_{j=-(2a-1)}^{2a-1} f_{j,BV} \] and $f_{j,BV}\in\Hom(V'_j,V_{j+1})$ for each $j$. 

\begin{rem}
   One may suspect that the aforementioned key fact relating Barbasch-Vogan duals and moment map correspondence of nilpotent orbits is not a coincidence, and it would be interesting to study this further, especially in the context of the \cite{BZSV} framework.
\end{rem}

Now recall that as in Theorem \ref{thm:SmoothMain}, the main result of \cite{GZ} is the following

\begin{thm}\label{thm:SmoothMainDegenerate}
The map   \begin{align*}
    q : \omega^\infty &\rightarrow \S(U', \overline{\chi_\gamma'}\bs G')\\
    \Phi &\mapsto \phi:= \big(g' \mapsto (g'\cdot\Phi)(f_{BV})\big)
\end{align*} 
is surjective, and induces the isomorphism $(\omega^\infty)_{U_{BV},\chi_{\gamma^\vee_{BV}}} \cong \S(U', \overline{\chi_\gamma'}\bs G') $.
\end{thm}

As in Section \ref{sec:PlancherelRecall}, we have a commutative square 

\[\begin{tikzcd}
& \omega^\infty \arrow[dl,"{q}"]  \arrow[rd, "{\theta_\sigma}"] &\\
(\omega^\infty)_{U_{BV},\chi_{\gamma^\vee_{BV}}} \cong \S(U', \overline{\chi_\gamma'}\bs G')\arrow[dr,"\alpha_{\sigma}"] & & \sigma\otimes\theta(\sigma) = \sigma\otimes\pi\arrow[ld,"{l_{\pi,BV}}"]\\
& \sigma & 
\end{tikzcd}\]

\noindent Here, $\sigma$ is any $\psi$-generic irreducible representation of $G'$, $\pi=\theta(\sigma)$ its theta lift to $G$, $\alpha_\sigma$ is (part of the family of) maps facilitating the spectral decomposition in the Whittaker-Plancherel theorem for $G'$, and most importantly, $l_{\pi,BV} \in \Hom_{U_{BV}}(\pi,\chi_{\gamma^\vee_{BV}})$ is \textit{defined so that the square is commutative}. These (canonical) local functionals $l_{\pi,BV}$ play the role of the `local relative characters', whose definition is one of the key issues highlighted in \cite[Remark 1.5]{MWZ}.

\vskip 5pt

With this set-up in place, the plan for the rest of the section is clear: we will compute the spherical local relative character as in Section \ref{sec:BasicFunction}, and compute the global constants as in Section \ref{sec:GlobalPeriod}. This will allow us to deduce \cite[Conjecture 1.4]{MWZ}.

\subsection{Spherical local relative character}

As in Section \ref{sec:BasicFunction}, let $\Phi_0\in \omega^\infty$ be the characteristic function of $Y_{(r),(r+1)}(\o_F)$. For each $\sigma$, fix a $G'(\o_F)$-invariant vector $v_0 \in \sigma$ with $\langle v_0,v_0\rangle_\sigma = 1$ and a $G(\o_F)$-invariant vector $w_0\in \theta(\sigma)$ with $\langle w_0,w_0\rangle_{\theta(\sigma)} = 1$. Suppose $\theta_\sigma(\Phi_0) = c_\sigma v_0\otimes w_0$. 

Now, as was recorded in (\ref{eqn:LocalPeriodRelation}), we have the local period relation 
\begin{align}\label{eqn:LocalPeriodRelationNew}
        \langle \alpha_{\sigma}(q(\Phi_0)), v_0\rangle_{\sigma} &=  \langle l_{\pi,BV}(\theta_\sigma(\Phi_0)), v_0\rangle_{\sigma}
    \end{align}
On one hand, we have on the left-hand-side of (\ref{eqn:LocalPeriodRelationNew}), 
\begin{align}\label{eqn:LocalPeriodRelationNewLeft}
        \langle \alpha_{\sigma}(q(\Phi_0)), v_0\rangle_{\sigma} &=  \langle q(\Phi_0), \beta_\sigma(v_0)\rangle_{U', \overline{\chi_\gamma'}\bs G'} \\
        &= l_\sigma(v_0) \cdot \vol(G'(\o_F)), \nonumber
    \end{align}
since, as in the remark at the start of Section \ref{sec:BasicFunction}, we have that $q(\Phi_0)$ is the basic function of $(U', \overline{\chi_\gamma'})\bs G'$, and, as in the remark after (\ref{eq:Local}), we have $l_{\sigma} = \text{ev}_{\id_{G'}} \circ \beta_{\sigma}$.

On the other hand, on the right-hand-side of (\ref{eqn:LocalPeriodRelationNew}), we have 
\begin{align}\label{eqn:LocalPeriodRelationNewRight}
         \langle l_{\pi,BV}(\theta_\sigma(\Phi_0)), v_0\rangle_{\sigma} &=  c_\sigma \cdot l_{\pi,BV}(w_0).
    \end{align}

\noindent Combining (\ref{eqn:LocalPeriodRelationNewLeft}) and (\ref{eqn:LocalPeriodRelationNewRight}), we get 
\begin{align}\label{eqn:LocalPeriodRelationNewLeftRight}
         l_{\pi,BV}(w_0) &=  \frac{l_\sigma(v_0) \cdot \vol(G'(\o_F))}{c_\sigma} .
    \end{align}

\noindent It follows that (all the quantities on the right-hand-side have already been computed in Section \ref{sec:BasicFunction})
\begin{align*}
         |l_{\pi,BV}(w_0)|^2 &=  \frac{|l_\sigma(v_0)|^2 \cdot \vol(G'(\o_F))^2}{|c_\sigma|^2} \\
         &= \frac{\zeta_F(k)\cdot \prod_{j=1}^a \zeta_F(2k-2j)}{L(1,\sigma,\Ad) \cdot L(k-a,\sigma,\std)}.
    \end{align*}

Multiplying with the earlier computed value of $|l_\pi(w_0)|^2 = |l_{\theta(\sigma)}(w_0)|^2$ (cf. (\ref{LocalUnramifiedComputation}), at the end of the proof of Theorem \ref{thm:GlobalPeriod}), we obtain the key result (observe here the cancellation of many factors!) 
\begin{align*}
         |l_{\pi,BV}(w_0)||l_\pi(w_0)| &=  \frac{1}{L(1,\sigma,\Ad)\cdot \vol(X(\o_F))}.
    \end{align*}

\vskip 5pt

\subsection{Global constants}

Retain all the notation and set-up of Section \ref{sec:GlobalPeriod} and Theorem \ref{thm:GlobalPeriod}. For $f$ a cuspidal automorphic form on $G$, denote the period integral \[ P_{\gamma^\vee_{BV}}(f) = \int_{[U_{BV}]} f(u) \cdot \overline{ \chi_{\gamma^\vee_{BV}}(u)  } \mathop{du}, \] taking as usual Tamagawa measure on all adelic groups. This is the ``degenerate Whittaker period" or ``degenerate Whittaker coefficient" of \cite{MWZ}.

Denote the global theta lift
    \begin{align*}
        A^{\text{Aut}}_\Sigma : \omega \otimes \overline{\Sigma} &\rightarrow \Pi \\
        \Phi \otimes \overline{f} &\mapsto \Big( g \mapsto \int_{[G']} \theta(\Phi)(gg')\overline{f(g')} dg' \Big).
    \end{align*}
 Then as in (\ref{eq:Global}), we have the global period relation 
 \begin{equation}\label{eq:GlobalNew}  
 P_{\gamma^\vee_{BV}} (A^{\text{Aut}}_\Sigma(\Phi,f)) = \langle q(\Phi), \beta^{\text{Aut}}_\Sigma(f) \rangle_{(U', \overline{\chi_\gamma'})\bs G'(\A)}, 
 \end{equation}
where here \begin{align*}
        \beta^{\text{Aut}}_\Sigma : \Sigma &\rightarrow C^\infty((U', \overline{\chi_\gamma'})\bs G'(\A)) \\
        f &\mapsto \Big( g' \mapsto P_{\gamma'}(g'\cdot f)  \Big).
    \end{align*}

Similarly, we have factorisations 

 \[ A_\Sigma^{\text{Aut}} = a(\Sigma) \cdot \prod_v^\ast A_{\Sigma_v}, \] \[  P_{\gamma^\vee_{BV}} = c_{BV}(\Pi) \cdot \prod_v^\ast l_{\Pi_v}, \]  for constants $c_{BV}(\Pi), a(\Sigma)$. As in \cite[Proposition 12.4]{GW2} (by the Rallis inner product formula), we have $|a(\Sigma)|=1$.   

Now on the left-hand-side of the global (\ref{eq:GlobalNew}), we have 
\begin{equation}\label{eq:GlobalNewLeft}  P_{\gamma^\vee_{BV}} (A^{\text{Aut}}_\Sigma(\Phi,f)) = c_{BV}(\Pi)\cdot a(\Sigma)\cdot \prod_v^\ast  l_{\Pi_v} (\langle \theta_v(\Phi_v), f_v\rangle), \end{equation}
while on the right-hand-side of (\ref{eq:GlobalNew}), we have
\begin{equation}\label{eq:GlobalNewRight}  \langle q(\Phi), \beta^{\text{Aut}}_\Sigma(f) \rangle_{(U', \overline{\chi_\gamma'})\bs G'(\A)} = c(\Sigma)\cdot \prod_v^\ast  \langle q_v(\Phi_v), \beta_{\Sigma_v}(f_v)\rangle_{U', \overline{\chi_\gamma'}\bs G'}. \end{equation}
Comparing with the local (\ref{eqn:LocalPeriodRelationNew}), it is clear that we should have
 \[ c_{BV}(\Pi)\cdot a(\Sigma) = c(\Sigma).\]
It follows that \[ |c_{BV}(\Pi)|^2 = |c(\Sigma)|^2 = \frac{1}{|S_\phi|}.\]

\subsection{} Putting the preceding computations together, we obtain 

\begin{thm}\label{thm:GlobalPeriodNewOnly}
Keeping the same set-up and notation as in Theorem \ref{thm:GlobalPeriod}, we have 
\[ |P_{\gamma^\vee_{BV}}(f)|^2 = \frac{1}{|S_\phi|} \Delta_G^T(1) \frac{1}{\prod_i L^T(i/2, \Sigma, V_X^{(i)})\cdot L^T(1, \Sigma, \Ad) } \prod_{v\in T} |l'_{\Pi_v,BV} (f_v)|^2 \]
for $f=\otimes_v f_v \in \otimes_v \Pi_v$ and where $l'_{\Pi_v,BV}$ are the (normalised as in \cite{GW2}, so that $l'_{\Pi_v,BV} (w_{0,v})=1$ for almost all $v$) canonical local functionals.
\end{thm}

\noindent This is precisely \cite[Conjecture 1.4]{MWZ} for the degenerate Whittaker period associated to $\gamma^\vee_{BV}$. We also obtain: 

\begin{thm}\label{thm:GlobalPeriodNew}
Keeping the same set-up and notation as in Theorem \ref{thm:GlobalPeriod}, we have 
 \[ |P_{\gamma^\vee_{BV}}(f)||P_\gamma(f)| = \frac{1}{|S_\phi|} \frac{\Delta_G^T(1)}{\Delta_L^T(1)} \frac{1}{L^T(1, \Sigma, \Ad) } \prod_{v\in T} |l'_{\Pi_v,BV} (f_v)||l'_{\Pi_v} (f_v)| \]
for $f=\otimes_v f_v \in \otimes_v \Pi_v$ and where $l'_{\Pi_v,BV},l'_{\Pi_v}$ are the (normalised as in \cite{GW2}, so that $l'_{\Pi_v,BV} (w_{0,v})=l'_{\Pi_v} (w_{0,v})=1$ for almost all $v$) canonical local functionals.
\end{thm}

\noindent This is precisely \cite[Conjecture 1.8]{MWZ}, obtained by multiplying \cite[Conjecture 1.4]{MWZ} for the degenerate Whittaker period, with the BZSV conjecture \cite[Conjecture 1.1]{MWZ}.

\vskip 5pt

Finally, we note that \cite[Conjecture 1.11]{MWZ} for the rank 1 spherical variety $\O_{2k-1} \backslash \O_{2k}$ can be viewed as a special case of the results in this section.  The techniques of this section (and more generally, this paper) will also similarly resolve \cite[Conjecture 1.11]{MWZ} for the rank 1 spherical variety $\O_{2k} \backslash \O_{2k+1}$, via the classical theta correspondence between the metaplectic groups and odd orthogonal groups. The other (exceptional) rank 1 spherical varieties are expected to be resolved via exceptional theta correspondences; there are, of course, more technical details that need to be worked out concerning the exceptional theta correspondences.  The recent preprint \cite{LW} treats the case of the spherical variety ${\rm Spin_9} \backslash F_4$.

\end{document}